\documentclass[12pt,letterpaper]{article}
\usepackage{graphicx}
\usepackage{amsmath}
\usepackage{amsfonts}
\usepackage{amsthm}
\usepackage{amssymb}
\usepackage{color}

\usepackage{verbatim} 
\usepackage{hyperref}

\newtheorem{prop}{Proposition}
\newtheorem{theorem}[prop]{Theorem}

\theoremstyle{definition}
\newtheorem{exercise}{Exercise}

\newcommand{\be}{\begin{equation}}
\newcommand{\ee}{\end{equation}}
\newcommand{\bes}{\begin{equation*}}
\newcommand{\ees}{\end{equation*}}
\newcommand{\bea}{\begin{eqnarray}}
\newcommand{\eea}{\end{eqnarray}}
\newcommand{\beas}{\begin{eqnarray*}}
\newcommand{\eeas}{\end{eqnarray*}}

\begin{document}

\title{Attacking ApSimon's Mints}

\author{Tanya Khovanova\\MIT}

\maketitle

\begin{abstract}
ApSimon's Mints problem is a very difficult and often misunderstood counterfeit-coin puzzle. I explain the problem and suggest ways to approach it, while giving several fun exercises for the reader.
\end{abstract}

Hugh ApSimon described the following coin puzzle in his book \textit{Mathematical Byways in Ayling, Beeling and Ceiling} \cite{A}.

\begin{quote}
New coins are being minted at $n$ independent mints. There is a suspicion that some mints might use a variant material for the coins. There can only be one variant material: fake coins weigh the same independently of the mint. The weight of genuine coins is known, but the weight of fake coins is not. There is a machine that can precisely weigh any number of coins, but the machine can only be used twice. You can request several coins from each mint and then perform the two weighings so that you can deduce with certainty which mints produce fake coins and which mints produce real coins. What is the minimum total of coins you need to request from the mints?
\end{quote}

I will follow ApSimon's notation. Suppose $P_r$ and $Q_r$ are the numbers of coins from the mint $r$ used in the first and the second weighing correspondingly, where $1 \leq r \leq n$. The goal is to provide the list of mints producing fake coins while minimizing $\Sigma_r \max(P_r,Q_r)$.  Let us denote by $W$ the weight of the genuine coin and by $W(1 + \epsilon)$ the weight of the fake coin. We do not know $\epsilon$, except that it is not zero.

Let $d_r$ be either 0 or 1, depending on what material the $r$-th mint uses. Thus, the coin from the $r$-th mint weighs $W(1 + d_r\epsilon)$. We know the results of these two weighings and the weight of the genuine coin. Therefore, we can calculate the following two values: $a = \Sigma_r P_r d_r \epsilon$ and $b = \Sigma_r Q_r d_r \epsilon$.

It is clear that we need to request at least one coin from each mint and use it in at least one weighing: $P_r$ + $Q_r > 0$. If both sums $a$ and $b$ are zero, then all the mints are producing genuine coins. Neither of the two values $a$ or $b$ gives us much information as we do not know $\epsilon$. We can get rid of $\epsilon$ by dividing $a$ by $b$.

There are $2^n - 1$ combinations of possible answers: these are subsets of the set of mints producing fake coins given that there is at least one. Thus we need to select numbers $P_r$ and $Q_r$, so that $a/b$ produces $2^n - 1$ possible answers for different sets of values of $d_r$.

Let us consider cases in which the total number of mints is small. If there is one mint we can take one coin and we will not even need a second weighing. For two mints we need one coin from each mint for a total of 2. 

\begin{exercise}\label{ex:3mints}
For three mints, one coin from each mint is not enough.
\end{exercise} 

It is possible to test three mints with four coins: one each from the first and second mints and two from the third mint. The coins from each mint for the first and second weighings are (0,1,2) and (1,1,0) respectively. 

To prove that this works we need to calculate $(d_2 + 2d_3)/(d_1 + d_2)$ for seven different combinations of $d_r$ and check that they are all different. 

\begin{exercise}
Check that the above weighings provide a solution.
\end{exercise} 

This puzzle seems to be very difficult \cite{GN}. We only know the answer if the number of mints is not more than seven. The corresponding sequence A007673 in the OEIS \cite{OEIS} is: 1, 2, 4, 8, 15, 38, 74. It is possible to give bounds for this sequence, but they are far, far apart. The lower bound is $n$. And the ApSimon's book offers a construction for two weighings where $P_r = r!$ and $Q_r = 1$, giving the upper bound of $\Sigma_r r!$. 

I would like to suggest two variations of the problem:

\begin{enumerate}
\item Minimize the sum of coins used in both weighings: $\Sigma_r P_r + \Sigma_r Q_r$.
\item Minimize the maximum of coins taken from a mint: $\max_r(P_r,Q_r)$.
\end{enumerate}

My next suggestion would be to formulate the ApSimon's Mints problem in terms of linear algebra. Consider $V_r = (P_r,Q_r)$ as a non-zero vector in a 2-dimensional plane. Suppose $K$ is a subset of the vectors. Denote $S_K$ the sum of the vectors in $K$. 

\begin{theorem}\label{thm:la}
The set of non-zero vectors $V_r$ with non-negative coefficients provides a solution to the ApSimon’s Mints problem if and only if for any two distinct subsets $K$ and $J$, the sums $S_K$ and $S_J$ are not collinear.
\end{theorem}

\begin{proof}
Suppose we are trying to differentiate whether the set of mints producing fake coins is $K$ or $J$. Collinearity of $S_K$ and $S_J$ means that the corresponding values of $a/b$ are the same for both sets.
\end{proof}

Now I would like to suggest another direction. Given the limited number of coins $c$ we can request from the mints, what is the maximum number of mints we can test?

Let us see what happens if the limit is 1. By Theorem~\ref{thm:la} all the vectors have to be different, so we are limited to the set (0,1), (1,0), and (1,1). But the sum of the first two vectors is collinear with the third. So we cannot test three mints. This answers Exercise~\ref{ex:3mints}. 

Let us see what happens if we limit the number of coins by 2. Then 8 possible vectors are (0,1), (0,2), (1,0), (1,1), (1,2), (2,0), (2,1), and (2,2). At most one vector from each pair $\{(0,1)$ and $(0,2)\}$, $\{(1,0)$ and $(2,0)\}$, and $\{(1,1)$ and $(2,2)\}$ can be present. So we are limited by 5 vectors: $(x,0)$, $(0,y)$, $(z,z)$, $(1,2)$, and $(2,1)$. Both vectors (1,2) and (2,1) cannot be together with $(z,z)$. 

\begin{exercise}
Show that it is impossible to test 4 mints when $c=2$.
\end{exercise} 

On the other hand, we already know a solution for 3 mints where the number of coins from each mint is limited by 2. It corresponds to vectors: $(0,1)$, $(1,1)$, and $(2,1)$.

The ApSimon's Mints problem seems intractable. The new variation might be easier to calculate and to estimate the bounds. The good news: the result of the calculations will provide bounds for the original ApSimon's Mints problem. Indeed, if we can find a solution for $n$ mints with the maximum number of coins equal to $c$, then there exists a solution for the original problem with a bound of $cn$.

I do not want to leave readers with puzzles that might end up being unyielding, so I suggest the following puzzle. 

\begin{exercise}
Suppose $\epsilon$ is known. Solve the ApSimon's Mints problem in this case: Minimize the total coins requested from mints, $\Sigma_r \max_r(P_r,Q_r)$.
\end{exercise} 

If this exercise is too easy, then generalize it to any number of weighings.

\begin{exercise}
Suppose $\epsilon$ is known. Solve the ApSimon's Mints problem for any given number of weighings: Minimize the total coins requested from mints, $\Sigma_r \max_r(P_r,Q_r)$.
\end{exercise}

\end{document}